\documentclass{article}

\usepackage{geometry}

\usepackage{amssymb}
\usepackage{amsmath}
\usepackage{amsthm}

\newtheorem{theorem}{Theorem}[section]
\newtheorem{conj}[theorem]{Conjecture}
\newtheorem{lemma}[theorem]{Lemma}
\newtheorem{claim}{Claim}[theorem]

\theoremstyle{definition}

\newcommand{\mcf}{\mathcal{F}}

\title{The extremal functions for triangle-free graphs with excluded minors}
\author{Robin Thomas \and Youngho Yoo}

\begin{document}

\centerline{\Large \bf THE EXTREMAL FUNCTIONS FOR TRIANGLE-FREE}
\smallskip
\centerline{{\Large\bf  GRAPHS WITH EXCLUDED MINORS}%
\footnote{
12 January  2018, revised 14 July  2018.
}}

\bigskip
\bigskip

\centerline{{\bf Robin Thomas}%
\footnote{Partially supported by NSF under Grant No.~DMS-1700157.}
}
\smallskip
\centerline{and}
\smallskip
\centerline{{\bf Youngho Yoo}}
\bigskip
\centerline{School of Mathematics}
\centerline{Georgia Institute of Technology}
\centerline{Atlanta, Georgia  30332-0160, USA}
\bigskip

\begin{abstract}
\noindent
We prove two results:
\begin{enumerate}
\item A graph $G$ on at least seven  vertices with a vertex $v$ such that $G- v$ is planar and  $t$ triangles satisfies $|E(G)| \leq 3|V(G)|- 9 + t/3$.
\item For $p=2,3,\ldots,9$, a triangle-free graph $G$ on at least $2p-5$ vertices with no $K_p$-minor satisfies $|E(G)|\leq (p-2)|V(G)| - (p-2)^2$.
\end{enumerate}
\end{abstract}

\section{Introduction}
All {\em graphs} in this paper are finite and simple. 
{\em Cycles} have no ``repeated" vertices.
A graph is a {\em minor} of another if the first can be obtained from a subgraph of the second by contracting edges. 
An {\em $H$-minor} is a minor isomorphic to $H$.
Mader~\cite{MadHom} proved the following beautiful theorem.

\begin{theorem}
\label{madertheorem}
For $p=2,3,\ldots,7$, a graph with no $K_p$-minor and  $V \geq p-1$ vertices has at most $(p-2)V - \binom{p-1}{2}$ edges.
\end{theorem}

For large $p$ however, a graph on $V$ vertices with no $K_p$-minor can have up to $\Omega(p \sqrt{\log p} V)$ edges as shown by several people  
(Kostochka~\cite{Kosminor,Koshad}, and
Fernandez de la Vega \cite{FerdlV} based on Bollob\'as, Catlin and Erd\"os \cite{BolCatErd}),. Already for $p=8,9$, there are $K_p$-minor-free graphs on $V$ vertices with strictly more than $ (p-2)V - \binom{p-1}{2}$ edges, but the exceptions are known. 
Given a graph $G$ and a positive integer $k$, we define \emph{$(G,k)$-cockades} recursively as follows. A graph isomorphic to $G$ is a $(G,k)$-cockade. Moreover, any graph isomorphic to one obtained by identifying complete subgraphs of size $k$ of two $(G,k)$-cockades is also a $(G,k)$-cockade,
and every $(G,k)$-cockade is obtained this way. The following is a theorem of J\o rgensen~\cite{Jor}.

\begin{theorem}
\label{jorgensentheorem}
A graph on $V\ge7$ vertices with no $K_8$-minor has at most  $ 6V - 21$ edges,  unless it is a $(K_{2,2,2,2,2},5)$-cockade.
\end{theorem}

\noindent
The next theorem is due to Song and the first author~\cite{SonThoK9}.

\begin{theorem}
\label{songthomastheorem}
A graph on $V\ge8$ vertices with no $K_9$-minor has at most $7V - 28$ edges, unless it is a $(K_{1,2,2,2,2,2},6)$-cockade or isomorphic to $K_{2,2,2,3,3}$.
\end{theorem}

\noindent 
The first author and Zhu~\cite{ThoZhu} conjecture the following generalization.

\begin{conj}
\label{co:k10}
A graph on $V\ge9$ vertices with no $K_{10}$-minor has at most $8V - 36$ edges, unless it is  isomorphic to one of the following graphs:
\begin{enumerate}
\item[\rm(1)]
 a $(K_{1,1,2,2,2,2,2},7)$-cockade,
\item[\rm(2)] $K_{1,2,2,2,3,3}$,
\item[\rm(3)] $K_{2,2,2,2,2,3}$,
\item[\rm(4)] $K_{2,2,2,2,2,3}$ with an edge deleted,
\item[\rm(5)] $K_{2,3,3,3,3}$,
\item[\rm(6)] $K_{2,3,3,3,3}$ with an edge deleted,
\item[\rm(7)] $K_{2,2,3,3,4}$, and 
\item[\rm(8)] the graph obtained from the disjoint union of $K_{2,2,2,2}$ and $C_5$ by adding all edges joining them.
\end{enumerate}
\end{conj}

McCarty and the first author  studied the extremal functions for \emph{linklessly embeddable graphs}: graphs embeddable in 3-space such that no two disjoint cycles form a non-trivial link. Robertson, Seymour, and the first author~\cite{RobSeyThoSachs} showed that a graph is linklessly embeddable if and only if it has no minor isomorphic to a graph in the \emph{Petersen family}, which consists of the seven graphs (including the Petersen graph) that can be obtained from $K_6$ by $\Delta Y$- or $Y\Delta$-transformations.
Thus, Mader's theorem implies that a linklessly embeddable graph on $V$ vertices has at most  $4V-10$ edges.
McCarty and the first author~\cite{McCTho} proved the following.

\begin{theorem}
\label{biplinkltheorem}
A bipartite linklessly embeddable graph on $V \geq 5$ vertices has at most  $3V - 10$ edges, unless it is isomorphic to $K_{3, V-3}$.
\end{theorem}

\noindent
In the same paper McCarty and the first author  made the following three  conjectures.

\begin{conj}
\label{trifreelinklconjecture}
A triangle-free linklessly embeddatble graph on $V \geq 5$ vertices has at most  $3V - 10$ edges, unless it is isomorphic to $K_{3,V-3}$.
\end{conj}

\noindent
As a possible approach to Conjecture~\ref{trifreelinklconjecture} McCarty and the first author proposed the following.

\begin{conj}
\label{trinumlinklconjecture}
A linklessly embeddable graph on $V \geq 7$ vertices with $t$ triangles has at most  $3V - 9 + t/3$ edges.
\end{conj}

\noindent 
The third conjecture of McCarty and the first author  is as follows.

\begin{conj}
\label{bipminorconjecture}
For $p=2,3,\ldots,8$, a bipartite graph on $V \geq 2p-5$ vertices with no $K_p$-minor has at most  $(p-2)V - (p-2)^2$ edges.
\end{conj}

\subsection{Our results}
We first give a partial result to Conjectures \ref{trifreelinklconjecture} and \ref{trinumlinklconjecture}. An \emph{apex graph} is a graph $G$ with a vertex $a$ such that $G-a$ is planar. All apex graphs are linklessly embeddable. We show that Conjectures~\ref{trifreelinklconjecture} and \ref{trinumlinklconjecture} hold for apex graphs:
\begin{theorem} \label{apextheorem}
A triangle-free apex graph on $V \ge5$ vertices has at most $3V - 10$ edges, unless it is isomorphic to $K_{3, V-3}$. 
Moreover, an apex graph on $V \ge7$ vertices with $t$ triangles has at most  $3V - 9 + t/3$ edges.
\end{theorem}

\noindent 
Let us remark that the assumption that $V\ge7$ is necessary: let $G$ be the graph obtained from $K_6$
by deleting a perfect matching. 
Then $G$ has six vertices, $12$ edges and eight triangles; thus $|E(G)|=12\not\le 35/3=3V-9+t/3$.

Our second result proves a generalization of Conjecture \ref{bipminorconjecture} to triangle-free graphs for values of $p$ up to 9:
\begin{theorem} \label{trifreeminorcor}
For $p=2,3,\ldots,9$, a triangle-free graph with no $K_p$-minor on $V \geq 2p-5$ vertices has at most  $(p-2)V-(p-2)^2$ edges.
\end{theorem}

We prove Theorem~\ref{apextheorem} in Section~\ref{sec:apex} and Theorem~\ref{trifreeminorcor} in Section~\ref{sec:trfree}.

\section{Proof of Theorem \ref{apextheorem}}
\label{sec:apex}
For an integer $V$, 
by $V^+$ we denote $\max\{V,0\}$,  and we define $\psi(V):=(7-V)^++(5-V)^+$.
We need the following lemma.

\begin{lemma}
\label{lem:Vs}
Let $V_1,V_2\ge2$ be  integers, and let $V=V_1+V_2-1$.Then
$$\max\{\psi(V_1),1\}+\max\{\psi(V_2),1\}\le \psi(V)+10$$
with equality if and only if $V\le5$.
\end{lemma}

\begin{proof}
Assume first that both $V_1,V_2$ are at most five. If $V\ge6$, then $V_1+V_2\ge7$ and we have 
$$\max\{\psi(V_1),1\}+\max\{\psi(V_2),1\}=\psi(V_1)+\psi(V_2)=7-V-1+17-(V_1+V_2)\le \psi(V)+9.$$
If $V\le5$, then
$$\max\{\psi(V_1),1\}+\max\{\psi(V_2),1\}=\psi(V_1)+\psi(V_2)=7-V-1+5-V-1+12= \psi(V)+10.$$

We may therefore assume  that say $ V_2\ge6$. Then
\begin{align*}
\max\{\psi(V_1),1\}+\max\{\psi(V_2),1\}&=\max\{(5-V_1)^++(7-V_1)^+,1\}+\max\{(7-V_2)^+,1\}\\
&\le 3+5+1=9,
\end{align*}
%
as desired.
\end{proof}

Let $G$ be an apex graph on $V $ vertices and $E$ edges  with a vertex $a$ such that $G-a$ is planar. 
Let $G^\circ  := G-a$ be embedded in the plane, and let $V^\circ:=|V(G^\circ)|$ and $E^\circ:=|E(G^\circ)|$. Note that $V = V^\circ + 1$ and $E = E^\circ + d(a)$. 

\subsection{Triangle-free case}
First suppose that $G$ is triangle-free and that  $V\ge5$. Then $N(a)$ is an independent set. As $G^\circ$ is triangle-free, planar and has 
at least three vertices, it follows from Euler's formula that $E^\circ \leq 2V^\circ - 4$, so
\begin{align*}
E &= E^\circ + d(a)
\leq 2V^\circ - 4 + d(a) 
= 2V-6 + d(a)
\end{align*}
If $d(a) \leq V-4$, then we are done. As $d(a) \leq V-1$, we just need to check 3 cases:
\begin{enumerate}
	\item $d(a)=V-1$.
	As $N(a)$ is independent, $G^\circ$ is the empty graph on $V-1$ vertices, so $E = V-1\le3V-10$,
since $V\ge5$.

	\item $d(a)=V-2$.
	Then $a$ is adjacent to all but one vertex $u$ in $G^\circ$. Since $d(u) \leq V-2$ and $N(a)$ is independent, it follows that 
$E \leq 2V-4\le3V-10$, unless $V=5$, in which case $E\le3V-10$, except when $G$ is isomorphic to $K_{2,3}$, as desired.

	\item $d(a)=V-3$.
	Then $a$ is adjacent to all but two vertices $u,v$ in $G^\circ$. Since $G^\circ$ is triangle-free and $N(a)$ is independent, if 
$u$ is adjacent to $v$, then $E^\circ \leq V^\circ - 1$, in which case $E=E^\circ+V-3\le 2V-5\le3V-10$, since $V\ge5$; 
and if $u$ is not adjacent to $v$, then $E^\circ \leq 2(V^\circ-2)$, in which case $E\le E^\circ+V-3\le 3V-9$, 
with equality  if and only if $G^\circ $ is isomorphic to $ K_{2,V-3}$ and $G$ is isomorphic to $ K_{3,V-3}$. 
\end{enumerate}
Therefore $E \leq 3V-10$, unless $G$ is isomorphic to $ K_{3,V-3}$, as desired.

\subsection{General case}
Now suppose that $G$ has 
$t$ triangles.  Let $t^\circ$ denote the number of triangular faces of $G^\circ$ and let $t_a$ denote the number of triangles of $G$ incident with  $a$. Let $t' = t^\circ + t_a$. Since $t' \leq t$, it would suffice to show that
\begin{align} \label{t'inequality}
E \leq 3V-9+t'/3
\end{align}
However, this inequality does not always hold. Consider a graph $G$ obtained from $K_{3,V-3}$ with bipartition $(\{a,b,c\},\{v_1,\dots,v_{V-3}\})$ by adding the edge $bc$ and any subset of the edges $\{v_1v_2,v_2v_3,\allowbreak \dots,\allowbreak v_{V-4}v_{V-3}\}$. This gives an apex graph,
where $G-a$ is planar, with $E = 3V-9+t'/3 + 1/3$, violating the inequality (\ref{t'inequality}). Let us call any graph isomorphic to  such a graph {\em exceptional}.
What we will show is that every graph $G$ on at least seven vertices satisfies (\ref{t'inequality}), unless $G$ is exceptional.
Note that this proves Theorem \ref{apextheorem}, since an exceptional  graph 
has at least two  triangles which are not counted in~$t'$, and hence satisfies the inequality in Theorem \ref{apextheorem}.

In fact, we prove a stronger statement, and for the sake of the inductive argument we allow graphs on fewer than seven vertices. 
Let $\mcf$ denote the set of faces of $G^\circ$. Define
\begin{align*}
\phi(G,a) :&= \frac{t_a}{3} - \sum_{f \in \mcf} \frac{|f|-4}{3}
= \frac{t_a}{3} + \frac{t^\circ}{3} - \sum_{\substack{f\in \mcf \\ |f| \geq 5}} \frac{|f|-4}{3}
\leq \frac{t'}{3}.
\end{align*}

\noindent 
We prove the following:
\begin{theorem}
Let $G,a,V,E$  be as before. and let $V\ge2$. Then
\begin{itemize}
\item[\rm(1)] if $G$ is exceptional, then $E = 3V-9+\phi(G,a)+1/3$.
\end{itemize}
Otherwise
\begin{itemize}
\item[\rm(2)] $E \leq 3V-9+\phi(G,a)+\psi(V)/3$,
\item[\rm(3)] if $G-a$ has at least one non-neighbour of $a$, then $E \leq 3V-9+\phi(G,a)+(7-V)^+/3$, and 
\item[\rm(4)] if $G-a$ has at least two non-neighbours of $a$, then $E \leq 3V-9+\phi(G,a)$.
\end{itemize}
\end{theorem}
\begin{proof}
%
We proceed by induction on $V+E$. If $V=2$ and $E=0$, then
$$E=0=6-9+4/3+(7-2)/3=3V-9+\phi(G,a)+(7-V)^+/3.$$
If $V=2$ and $E=1$, then
$$E=1=6-9+4/3+(7-2)/3+(5-2)/3=3V-9+\phi(G,a)+\psi(V)/3.$$
We may therefore assume that $V\ge3$ and that the theorem holds for all graphs 
$G'$ with $|V(G')|+|E(G')|<V+E$.
We suppose for a contradiction that the theorem does not hold for $G$.
It follows that $G$ is not exceptional, because exceptional graphs satisfy the theorem.
Let $G^\circ:=G-a$, $V^\circ$ and $E^\circ$ be   as before.

\begin{claim} \label{bridgelessclaim}
The graph 
$G^\circ$ has no cut-edges.
\end{claim}
\begin{proof}
Suppose $e=xy$ is a cut-edge of $G^\circ$ incident with  a face $f_e$.
Let $C_1$ be the connected component of $G^\circ - e$ containing $x$, and let $C_2 = G^\circ -V( C_1)$. 
Define $G_i:= G[V(C_i) \cup \{a\}]$,  $V_i:=|V(G_i)|$ and $E_i:=|E(G_i)|$ for $i=1,2$.  Let $\mcf_i$ denote the set of faces of $C_i$, let 
$f_i$ denote the face of $C_i$ that contains $f_e$, and let $t_{a,i}$ denote the number of triangles incident with $a$ in $G_i$ for $i=1,2$. 
Then  $V=V_1+V_2-1$, $|f_e|=|f_1|+|f_2| + 2$, $\mcf = ((\mcf_1 \cup \mcf_2) \setminus \{f_1,f_2\}) \cup \{f_e\}$, 
and $t_{a,1}+t_{a,2} \leq t_a-\epsilon$, where $\epsilon =1$ if $a$ is adjacent to every vertex of $G-a$ and $\epsilon =0$ otherwise, and so
\begin{align*}
\phi(G_1,a)+\phi(G_2,a) &= \frac{t_{a,1}+t_{a,2}}{3} - \sum_{f \in \mcf_1}\frac{|f|-4}{3} - \sum_{f \in \mcf_2}\frac{|f|-4}{3} \\
&\leq \frac{t_a}{3} -\frac\epsilon 3- \left(\sum_{f \in \mcf}\frac{|f|-4}{3}\right) + \frac{|f_e|-4}{3} - \frac{|f_1|+|f_2|-8}{3}\\
&= \phi(G,a)+2-\epsilon/3.
\end{align*}
By Lemma~\ref{lem:Vs} $\max\{\psi(V_1),1\}+\max\{\psi(V_1),1\}\le\psi(V)+10$, with equality if and only if $V\le5$.
Note that $E=E_1+E_2+1$. 
By the induction  hypothesis each $G_i$ satisfies 
$$E_i \leq 3V_i - 9 + \phi(G_i,a) + {\max\{\psi(V_i),1\}}/{3},$$
where for $V_i\le4$ equality holds only if $a$  is adjacent to every vertex of $G_i-a $;
 thus 
\begin{align*}
E &= E_1+E_2+1 \\
&\leq 3(V_1+V_2) - 18 + \phi(G_1,a)+\phi(G_2,a) +(\max\{\psi(V_1),1\}+\max\{\psi(V_2),1\}) /3 +1\\
&\leq 3V - 15 + \phi(G,a)+3+(\psi(V)+10-\epsilon )/3. 
\end{align*} 
It follows that $E\le3V-9 +\phi(G,a)+\psi(V)/3$, a contradiction,  because if equality holds in the two inequalities above, then 
$V\le5$, which implies that $V_1,V_2\le4$, and hence $a$ is adjacent to every vertex of $G-a$, and consequently $\epsilon =1$.
This proves the claim in the case when either $V\ge7$  or $a$ is adjacent to every vertex of $G-a$.

We may therefore assume that $V\le6$ and that $a$ is not adjacent to every vertex of $G-a $.
Assume next that $a$ is adjacent to all but one vertex of $G-a $. By the symmetry we may assume that $a$  
is adjacent to every vertex of $G_1-a$ and all but one vertex of $G_2-a$.
Then
$$\psi(V_1)+(7-V_2)^+=7-V-1+12-V_1\le(7-V)^++9,$$
and hence 
\begin{align*}
E &= E_1+E_2+1 \\
&\leq 3(V_1+V_2) - 18 + \phi(G_1,a)+\phi(G_2,a) +(\psi(V_1)+(7-V_2)^+) /3 +1\\
&\leq 3V - 15 + \phi(G,a)+3+((7-V)^++9)/3 \\
&\le 3V-9 +\phi(G,a)+(7-V)^+/3,
\end{align*} 
a contradiction.

We may therefore assume that $a$ is not adjacent to at least two vertices of $G-a $.
Assume next that $a$ is not adjacent to at least two vertices vertices of $G_2-a$. Then 
\begin{align*}
E &= E_1+E_2+1 \\
&\leq 3(V_1+V_2) - 18 + \phi(G_1,a)+\phi(G_2,a) +\psi(V_1) /3 +1\\
&\leq 3V - 15 + \phi(G,a)+3+8/3 \\
&\le 3V-9 +\phi(G,a),
\end{align*} 
a contradiction.

We may therefore assume that $a$ is not adjacent to exactly one  vertex of $G_i-a $ for $i=1,2$. We have 
\begin{align*}
E &= E_1+E_2+1 \\
&\leq 3(V_1+V_2) - 18 + \phi(G_1,a)+\phi(G_2,a) +((7-V_1)^++(7-V_2)^+) /3+1 \\
&\leq 3V - 15 + \phi(G,a)+3+10/3 \\
&= 3V-9 +\phi(G,a)+1/3,
\end{align*} 
with equality if and only if $V_1=V_2=2$, in which case $G$ is exceptional, in either case a contradiction.
Thus the claim holds.
\end{proof}

\begin{claim} \label{naindepclaim}
We have that 
$t_a=0$; that is, $N(a)$ is independent. In particular, $\phi(G,a) = \sum_{f\in \mcf}\frac{4-|f|}{3}$.
\end{claim}
\begin{proof}
Suppose there exist adjacent vertices $x,y \in N(a)$ 
As $xy$ is not a cut-edge of $G^\circ$ by Claim~\ref{bridgelessclaim}, it is incident with two distinct faces $f_1,f_2$. Let $G'  := G - xy$,
let $E':=|E(G')|$ and let $f'$ be the new face obtained in $G^\circ-xy$. Let $\mcf'$ denote the set of faces of $G^\circ-xy$ and let $t'_a$ denote the number of triangles of $G'$ incident with  $a$. Then $t_a' = t_a-1$, $|f_1|+|f_2| = |f'|+2$, and $\mcf = (\mcf' \setminus \{f'\})\cup \{f_1,f_2\}$, so
\begin{align*}
\phi(G',a) &= \frac{t'_a}{3} - \sum_{f \in \mcf'}\frac{|f|-4}{3} \\
&= \frac{t_a}{3} - \frac{1}{3} - \left(\sum_{f \in \mcf}\frac{|f|-4}{3}\right) + \frac{|f_1|+|f_2|-8}{3} - \frac{|f'|-4}{3} \\
&= \phi(G,a) - 1
\end{align*}

Since $G$ does not satisfy  the theorem, it is not exceptional, and 
so neither is $G'$. Let $x:=\psi(V)$ if $a$ is adjacent to every vertex of $G-a $,
let $x:=(7-V)^+$ if $a$ is adjacent to all but one vertex of $G-a$  and let $x:=0$ otherwise. By the induction hypothesis
\begin{align*}
E &= E' + 1 \\
&\leq 3V - 9 + \phi(G',a) + 1 +x\\
&= 3V - 9 + \phi(G,a)+x,
\end{align*}
a contradiction. 
\end{proof}

\begin{claim}
\label{cl:noisolate}
The graph $G^\circ$ has no isolated vertices.
\end{claim}

\begin{proof}
Suppose for a contradiction that $v$ is an isolated vertex of $G^\circ$.
Let $G' = G-v$, let $V':=|V(G')|$ and  let $E'=|E(G')|$.
Then $\phi(G',a) = \phi(G,a)$.
Let $x':=\psi(V')$ and $x:=\psi(V)$ if $a$ is adjacent to every vertex of $G-a $,
let $x':=(7-V')^+$ and $x:=(7-V)^+$ if $a$ is adjacent to all but one vertex of $G-a$  and let $x=x':=0$ otherwise. If $v$ is adjacent to $a$,
then  by the induction hypothesis
\begin{align*}
E &= E' + 1 \\
&\leq 3V' - 9 + \phi(G',a) + 1 +\max\{x',1\}/3\\
&\le 3V-3 - 9 + \phi(G,a)+1+x/3+2/3\\
&\le 3V - 9 + \phi(G,a)+x/3,
\end{align*}
and if $v$ is not adjacent to $a$,
then
\begin{align*}
E &= E' \leq 3V' - 9 + \phi(G',a)  +\max\{\psi(V'),1\}/3\\
&\le 3V -3- 9 + \phi(G,a)+8/3\\
&\le 3V - 9 + \phi(G,a),
\end{align*}
a contradiction in either case. 
\end{proof}

\begin{claim} \label{nadegclaim}
If $v \in N(a)$, then $v$ has at least three neighbours in $G^\circ$; that is, $d(v) \geq 4$.
\end{claim}
\begin{proof}
Since $G^\circ$ has no cut-edges by Claim~\ref{bridgelessclaim} and no isolated vertices by Claim~\ref{cl:noisolate}, 
$v$ has at least two neighbours in $G^\circ$. Suppose it has exactly two neighbours, and let $f_1,f_2$ be the two faces of $G^\circ$ incident to $v$. Let $G' = G-v$, let $V':=|V(G')|$, let $E'=|E(G')|$ and let $f'$ denote the new face in $G^\circ - v$. 
Then $|f_1|+|f_2|=|f'|+4$, and
\begin{align*}
\phi(G',a) &= -\sum_{f \in \mcf'} \frac{|f|-4}{3} \\
&= -\left(\sum_{f \in \mcf} \frac{|f|-4}{3}\right) + \frac{|f_1|+|f_2|-8}{3} - \frac{|f'|-4}{3} \\&= \phi(G,a)
\end{align*}
As $G$ does not satisfy the theorem, it is not exceptional, and hence neither is $G'$.
Furthermore, the neighbours of $v$ are not adjacent to $a$ by Claim~\ref{naindepclaim}, and so by the induction hypothesis
\begin{align*}
E &= E'+3 \\
&\leq 3V' - 9 + \phi(G',a) + 3\\
&= 3V - 9 + \phi(G,a),
\end{align*}
a contradiction.
\end{proof}

We now show an upper bound on the degree of $a$ by a simple discharging argument. Start by assigning a charge of one to each vertex in $N(a)$, and for each $v \in N(a)$ distribute its charge equally to its incident faces in $G^\circ$. Then the sum of the charges of the faces of $G^\circ$ is equal to $d(a)$.

By Claim \ref{nadegclaim}, each $v \in N(a)$ is incident to at least three  faces, so it gives at most 1/3 charge to each incident face. By Claim \ref{naindepclaim}, each face $f \in \mcf$ is incident to at most $\lfloor |f|/2\rfloor$ neighbours of~$a$. Thus the final charge of face $f$ is at most $\lfloor |f|/2\rfloor/3$, and
\begin{align*}
d(a) \leq \sum_{f \in \mcf} \frac{\lfloor |f|/2\rfloor}{3}
\end{align*}
Since $\lfloor k/2 \rfloor \leq k-2$ for all $k\geq 3$,
\begin{align} \label{dabound}
d(a) \leq \sum_{f \in \mcf} \frac{|f|-2}{3}
\end{align}

The remainder of the proof follows from arithmetic using Euler's formula. Let $F^\circ$ denote the number of faces of $G^\circ$.
By the handshaking lemma, we have $2E^\circ = \sum_{f \in \mcf} |f|$. Since $F^\circ = \sum_{f \in \mcf} 1$, by Euler's formula:
\begin{align*}
8 &\le 4V^\circ - 4E^\circ + 4F^\circ \\ 
&= 4V^\circ - 2E^\circ - \sum_{f \in \mcf} (|f| - 4)
\end{align*} 
Rearranging, we have
\begin{align}\label{euler1}
E^\circ \le 2V^\circ - 4  - \sum_{f \in \mcf} \frac{|f|-4}{2}
\end{align}
Similarly, we have $3F^\circ \leq 2E^\circ \le 2V^\circ + 2F^\circ - 4$, which gives
\begin{align} \label{euler2}
F^\circ \leq 2V^\circ - 4.
\end{align}
Putting (\ref{dabound}), (\ref{euler1}), and (\ref{euler2}) together, we have
\begin{align*}
E &= E^\circ + d(a) \\
&\leq V^\circ + F^\circ - 2 + \sum_{f\in \mcf} \frac{|f|-2}{3} \\
&= V^\circ + \frac{1}{3}F^\circ + \frac{2}{3}E^\circ - 2 \\
&\leq V^\circ + \frac{1}{3}(2V^\circ - 4) + \frac{2}{3}\left(2V^\circ - 4 - \sum_{f \in \mcf}\frac{|f|-4}{2}\right) - 2 \\
&= 3V^\circ - 6 - \sum_{f \in \mcf}\frac{|f|-4}{3} \\
&= 3V - 9 + \phi(G,a),
\end{align*}
a contradiction.
\end{proof}

\section{Proof of Theorem \ref{trifreeminorcor}}
\label{sec:trfree}
We prove the following slightly more general statement from which Theorem \ref{trifreeminorcor} follows:

\begin{theorem}\label{trifreeminortheorem}
Let $p \ge4$ be an integer. Suppose that no  graph $G$ with $|E(G)| > (p-2)|V(G)| - \binom{p-1}{2}$ can be obtained by contracting $\max\{p-4,2\}$
 edges from a triangle-free graph on at least $2p-3$ vertices with no $K_p$-minor. Then every triangle-free  graph on $V \geq 2p-5$ 
vertices with no $K_p$-minor has at most  $ (p-2)V - (p-2)^2$ edges.
\end{theorem}

Let us first show that Theorem \ref{trifreeminortheorem} implies Theorem \ref{trifreeminorcor}:
\begin{proof}[Proof of Theorem \ref{trifreeminorcor}, assuming Theorem~\ref{trifreeminortheorem}]For $p=2,3$ Theorem~\ref{trifreeminorcor} is easy.
For $p=4,5,6,7$, it follows directly from Theorems \ref{madertheorem} and~\ref{trifreeminortheorem}, as there are no  graphs $G$ on at least $p-1$ vertices  with 
no $K_p$-minor and strictly more than $(p-2)|V(G)|-\binom{p-1}{2}$ edges.

For $p=8$, by Theorem \ref{jorgensentheorem}, a  graph $G$ on at least seven vertices 
with no $K_8$-minor and strictly more than $ 6|V(G)|-21$ edges 
is a $(K_{2,2,2,2,2},5)$-cockade. It is easy to see that, given any four vertices of a $(K_{2,2,2,2,2},5)$-cockade, one can always find a triangle disjoint from those four vertices. Thus a $(K_{2,2,2,2,2},5)$-cockade cannot be obtained by contracting four edges from a triangle-free graph, and the result follows by 
Theorem~\ref{trifreeminortheorem}.

For $p=9$, by Theorem \ref{songthomastheorem}, a  graph $G$ on at least eight vertices 
with no $K_9$-minor and strictly more than $ 7|V(G)| - 28$ edges is either a $(K_{1,2,2,2,2,2},6)$-cockade or isomorphic to $K_{2,2,2,3,3}$. Again it is easy to verify that, given any five vertices of such a graph, one can always find a triangle disjoint from those five vertices. Therefore neither a $(K_{1,2,2,2,2,2},6)$-cockade nor $K_{2,2,2,3,3}$ can be obtained by contracting five edges from a triangle-free graph, and the result follows by Theorem \ref{trifreeminortheorem}.
\end{proof}

Let us remark that the same argument shows that Conjecture~\ref{co:k10} and 
Theorem~\ref{trifreeminortheorem} imply that Theorem~\ref{trifreeminorcor} holds for $p=10$,
formally as follows:

\begin{theorem}
If Conjecture~\ref{co:k10} holds, then every triangle-free  graph on $V \geq 15$ 
vertices with no $K_{10}$-minor has at most  $ 8V - 64$ edges.
\end{theorem}

\subsection{Proof of Theorem \ref{trifreeminortheorem}}
Let $p\ge4$ be an integer and let $G$ be a  counterexample with $|V(G)|$ minimum. Let $V=|V(G)|$ and $E=|E(G)|$.
We prove by a series of claims that $G$ is a complete bipartite graph. This leads to a contradiction: suppose $G$ is isomorphic to $K_{n, V-n}$ with $n \leq V/2$. If $n \geq p-1$, then $G$ contains a $K_p$-minor, and if $n \leq p-2$, then $E= n(V-n) \leq (p-2)(V-(p-2))$ as $V \geq 2p-5$.

\begin{claim}
\label{cl:noverts}
The graph $G$ has at least $2p-3$ vertices.
\end{claim}

\begin{proof}
If $V\le2p-4$, then by Mantel's theorem~\cite{Man} $E\le (p-2)(V-p+2)$, contrary to $G$ being a counterexample.
\end{proof}

\begin{claim}
\label{cl:mindeg}
$\delta(G) > p-2$
\end{claim}
\begin{proof}
Let $v$ be a vertex of $G$ of minimum degree, and let $G' = G-v$. Then $|E(G')| = E - \delta(G)$ and $|V(G')| = V-1$. Since $G$ is a minimal counterexample and $V>2p-3$ by Claim~\ref{cl:noverts},
\begin{align*}
(p-2)V - (p-2)^2 &< E \\
&=| E(G')| + \delta(G) \\
&\leq (p-2)|V(G')| - (p-2)^2 + \delta(G) \\
&= (p-2)V - (p-2)^2 +\delta(G)-(p-2),
\end{align*}
and so $p-2 < \delta(G)$, as desired.
\end{proof}

\begin{claim} \label{disjointedgesclaim}
For $1\leq k \leq p-2$, given any set of $k$ disjoint edges $\{e_1,\dots,e_k\}$ in $G$, we can find another edge disjoint from each $e_i$, $1\leq i \leq k$. 
\end{claim}
\begin{proof}
Let $e_1,\dots,e_k$ be given, where $e_i = x_iy_i$.  By Claim~\ref{cl:noverts} there exists a vertex $v$ not equal to any $x_i,y_i$.  Since $G$ is triangle-free, $v$ can be adjacent to at most one of $\{x_i,y_i\}$ for each $i$. Since $deg(v) \geq \delta(G) > p-2 \geq k$ by Claim~\ref{cl:mindeg}, 
there is an edge incident with  $v$ disjoint from each $e_i$, as desired.
\end{proof}

\begin{claim} \label{4cycleclaim}
Let $e_1,e_2$ be any two disjoint edges of $G$. Then $G[e_1\cup e_2]$ forms a 4-cycle.
\end{claim}
\begin{proof}
Since $G$ is triangle-free, there are at most two edges between $e_1$ and $e_2$, and if there are two edges, $G[e_1\cup e_2]$ forms a 4-cycle. Suppose for a contradiction that there is at most one edge between $e_1$ and $e_2$.
Let $k=\max\{p-4,2\}$.
By Claim \ref{disjointedgesclaim}, we can find pairwise disjoint edges $e_3,\dots, e_{k}$, each disjoint from both $e_1$ and $e_2$. Let $G'$ be the graph obtained by contracting all edges $e_1,\dots,e_{k}$ and let $\ell$ denote the number of parallel edges identified.  
Then $|E(G')| = E-k-\ell$, $|V(G')| = V-k$, and $|E(G')| \leq (p-2)|V(G')| - \binom{p-1}{2}$ by hypothesis as $G'$ is obtained from the graph $G$ by contracting $k$ edges.
Since there are $\binom{k}{2}$ pairs of edges in $\{e_1,\dots,e_{k}\}$ and there is at most one edge between $e_1$ and $e_2$, we have $\ell \leq \binom{k}{2}-1$. If $p\le5$, then let $\epsilon=1$; otherwise let $\epsilon=0$. Then
\begin{align*}
E &= |E(G')| + \ell + k  \\
&\leq \left((p-2)|V(G')| - \binom{p-1}{2}\right) + \left(\binom{k}{2} -1\right) + k \\
&=  (p-2)(V-k)  - \frac{(p-1)(p-2) - k(k-1)}{2} + k - 1 \\
&= (p-2)V - (p-2)^2-\epsilon\\
&\le (p-2)V - (p-2)^2,
\end{align*}
a contradiction since $G$ is a counterexample. 
\end{proof}

\begin{claim}
$G$ is a complete bipartite graph.
\end{claim}
\begin{proof}
Let $e = xy$ be an edge and let $v \in V(G) \setminus \{x,y\}$. By Claims~\ref{cl:mindeg} and~\ref{4cycleclaim}, $v$ is adjacent to either $x$ or $y$, but not both as $G$ is triangle-free. Thus $V(G) \setminus \{x,y\}$ can be partitioned into two disjoint sets $X' \cup Y'$ where every vertex in $X'$ is adjacent to $y$ and every vertex in $Y'$ is adjacent to $x$. Since $G$ is triangle-free, there are no edges between vertices of $X'$ and between vertices of $Y'$. Thus $G$ is bipartite with bipartition $X \cup Y$, where $X = X' \cup \{x\}$ and $Y= Y' \cup \{y\}$. Moreover, for any $x' \in X'$ and $y' \in Y'$, the two edges $xy'$ and $x'y$ induce a 4-cycle by Claim \ref{4cycleclaim}. Therefore $x'$ is adjacent to $y'$, completing the proof of the claim.
\end{proof}

\baselineskip 11pt
\vfill
\noindent
This material is based upon work supported by the National Science Foundation.
Any opinions, findings, and conclusions or
recommendations expressed in this material are those of the authors and do
not necessarily reflect the views of the National Science Foundation.

\end{document}